\documentclass{amsart}
\usepackage{math}
\usepackage{hyperref}
\newcommand{\Pairs}{\mathbf{P}}
\newcommand{\tWPQn}{\tilde{W}_{n}}
\begin{document}
\title{Two-dimensional cycle classes on $\Mbar_{0,n}$}
\date{\today}
\author{Rohini Ramadas and Rob Silversmith}
\maketitle

\begin{abstract}
  For each $n\ge5$, we give an $S_n$-equivariant basis for
  $H_4(\Mbar_{0,n},\Q)$, as well as for $H_{2(n-5)}(\Mbar_{0,n},\Q)$. Such a basis exists for
  $H_2(\Mbar_{0,n},\Q)$ and for $H_{2(n-4)}(\Mbar_{0,n},\Q)$, but it is not known whether one exists for
  $H_{2k}(\Mbar_{0,n},\Q)$ when $3\le k\le n-6$.
\end{abstract}

\section{Introduction}
%Should mention:
%\begin{itemize}
%\item What is a permutation basis?
%\item Main result
%\item Known things about homology of $\Mbar_{0,n}$ (details in next
%  section). Yuzvinsky basis?
%\item Bergstrom-Minabe
%\item Castravet-Tevelev
%\item Poincare duality
%\item What we don't know?
%\end{itemize}

The moduli space $\Mbar_{0,n}$ is an $(n-3)$-dimensional smooth
projective variety that parametrizes stable $n$-marked genus-zero
curves. $\Mbar_{0,n}$ admits a stratification whose intricate
combinatorics are reflected in the structure of its homology
groups. These groups have been described, in terms of generators and
relations, first by Keel \cite{Keel1992} and then by Kontsevich and Manin
\cite{KontsevichManin1994}. The groups were described further by Kapranov \cite{Kapranov1993}, Fulton-MacPherson \cite{FultonMacpherson1994}, Manin \cite{Manin1995}, Getzler \cite{Getzler1995}, Yuzvinsky \cite{Yuzvinsky1997}
and others. 

$\Mbar_{0,n}$ carries a natural $S_n$-action; this induces
$S_n$-actions on the homology groups of $\Mbar_{0,n}$. Getzler \cite{Getzler1995} gave an algorithm for computing the character of $H_{2k}(\Mbar_{0,n},\Q)$ as an $S_n$-representation, and Bergstr\"om and
Minabe \cite{BergstromMinabe2013} later gave a recursive formula for the character using spaces of weighted stable curves. 

% It was constructed by Knudsen \cite{} as a compactification of $\M_{0,n}$, which parametrizes configurations of $n$ labeled points on $\P^1$. The boundary $\Mbar_{0,n}\setminus\M_{0,n}$ is a simple normal crossings divisor with a rich combinatorial stratification by $k$-fold intersections of irreducible divisors in the boundary. This combinatorial structure is reflected in the Chow and homology groups of $\Mbar_{0,n}$. Keel \cite{Keel1992} gave a presentation of the Chow ring of $\Mbar_{0,n}$; this is generated by the classes of the irreducible divisors in the boundary. Keel also showed that the Chow group $A_k(\Mbar_{0,n})$ is isomorphic to the homology group $H_{2k}(\Mbar_{0,n},\Z)$, and is a finitely generated free $\Z$-module generated, though not freely, by the classes of the $k$-dimensional boundary strata. Kontsevich-Manin \cite{KontsevichManin} have given a generating set for the relations among $k$-dimensional boundary strata. Keel, Yuzvinsky, and others (?) have constructed, recursively, bases for $A_k(\Mbar_{0,n})/H_{2k}(\Mbar_{0,n},\Z)$. These constructions yield recursive formulas for the betti numbers of $\Mbar_{0,n}$. 
% 

Recall that a finite-dimensional representation $V$ of a finite group $G$ is called a \textit{permutation representation} if $V$ has a \textit{permutation basis}; that is, a basis $B$ of $V$ such that the action of $G$ on $V$ restricts to an action on $B$. 
% Not every representation has a permutation basis, for example, the
% sign representation of $S_n$ does not. 
We ask:

\begin{question}\label{q:MainQ}
  Is $H_{2k}(\Mbar_{0,n},\Q)$ a permutation representation of $S_n$
  for all $n\ge3$ and $k\ge0$?
\end{question}

Farkas and Gibney \cite{FarkasGibney2003} gave an affirmative answer if $k=n-4$ by producing a permutation basis of divisors, the case
$k=n-4$. Since Poincar\'e duality induces $S_n$-equivariant
isomorphisms $H_{2k}(\Mbar_{0,n},\Q)\to H_{2(n-3-k)}(\Mbar_{0,n},\Q)$,
this also implies an affirmative answer if $k=1.$ The first author found \cite{RamadasThesis}
a permutation basis for $H_{2}(\Mbar_{0,n},\Q)$ that is not Poincar\'e
dual to the basis of Farkas and Gibney --- interestingly, the two bases
are not isomorphic as $S_n$-sets for even $n$.

We give an
affirmative answer in the case $k=2$ (and therefore also $k=n-5$):

\medskip

\noindent \textbf{Theorem \ref{thm:MainThmWeak}.}  \textit{Let
$\Pairs^1_{2,n}$ denote the $S_n$-set of all subsets
$A\subseteq\{1,\ldots,n\}$ such that $5\le\abs{A}\le n$ and $\abs{A}$
is odd. Let $\Pairs^2_{2,n}$ denote the $S_n$-set of unordered pairs
$\{P_1,P_2\},$ where $P_1$ and $P_2$ are disjoint subsets of
$\{1,\ldots,n\}$ of cardinality at least 3. Then $H_4(\Mbar_{0,n},\Q)$ is a permutation representation of $S_n$, with a permutation basis in equivariant bijection with $\Pairs^1_{2,n}\sqcup\Pairs^2_{2,n}$, for all $n\ge3.$}

\medskip

% Poincar\'e duality induces $S_n$-equivariant isomorphisms
% $H_{2k}(\Mbar_{0,n},\Q)\to H_{2(n-3-k)}(\Mbar_{0,n},\Q)$, so Theorem
% \ref{thm:MainThmWeak} implies an affirmative answer to Question
% \ref{q:MainQ} when $k=n-5$. 

% It is not known whether
% $H_{2k}(\Mbar_{0,n},\Q)$ is a permutation representation of $S_n$ for
% general $k$,

For general $k$, Question \ref{q:MainQ} appears to be open. The known bases of
$H_{2k}(\Mbar_{0,n},\Q)$ are
constructed via recursion on $n$, which involves treating the $n$-th marked point as special. For example, Kapranov \cite{Kapranov1993} and Yuzvinsky \cite{Yuzvinsky1997} described bases (recursively and in closed form, respectively) for $H_{2k}(\Mbar_{0,n},\Q)$ that are permutation bases with respect to the $S_{n-1}$ action induced by permuting all but the $n$th point.
In contrast, a permutation basis
for $H_{2k}(\Mbar_{0,n},\Q)$, if it exists, would treat the $n$ marked
points symmetrically.

Very recent work of Castravet and Tevelev \cite{CastravetTevelev2020} on the derived category of $\Mbar_{0,n}$
implies that $H_*(\Mbar_{0,n},\Q)=\bigoplus_{k=0}^{n-3}H_{2k}(\Mbar_{0,n},\Q)$ is a
permutation representation; however, this does not answer Question \ref{q:MainQ}, as the permutation basis given for $H_*(\Mbar_{0,n},\Q)$ consists of Chern characters of certain vector bundles, which are not of pure degree.

% Farkas and Gibney \cite{FarkasGibney} have given a permutation basis
% of the divisor class group $H_{2(n-4)}(\Mbar_{0,n},\Q)$ --- Poincar\'e
% duality then yields a permutation basis of
% $H_{2}(\Mbar_{0,n},\Q)$. The first author found \cite{RohiniThesis} a
% different permutation basis for
% $H_{2}(\Mbar_{0,n},\Q)$. Interestingly, for even $n$, the two bases
% are nonisomorphic as $S_n$-sets.
% while
% the permutation representations described by Farkas-Gibney on the one
% hand and in \cite{RohiniThesis} on the other must be isomorphic, the
% actions on the two different permutation bases are only isomorphic for
% odd $n$.
%We are left with the question:

\subsection{An \texorpdfstring{$S_n$}{S\_n}-equivariant filtration of \texorpdfstring{$H_{2k}(\Mbar_{0,n})$}{H\_{2k}(M\_{0,n}-bar)} and the proof of Theorem \ref{thm:MainThmWeak}}
 Theorem \ref{thm:MainThmWeak} follows from the stronger Theorem \ref{thm:MainThmStrong}, which we now explain. In \cite{RamadasThesis}, the first author gave (Definition \ref{eq:RamadasFiltration}) an $S_n$-equivariant decomposition 
\begin{align}\label{eq:HDecomp}
    H_{2k}(\Mbar_{0,n},\Q)\cong\bigoplus_{r=1}^{\min\{k,n-2-k\}}Q^r_{k,n}.
\end{align}
The first author also showed (Theorem \ref{thm:RohiniQ1kn}) that for all $k\ge0$ and $n\ge3$, $Q^1_{k,n}$ has a permutation basis, in equivariant bijection with the $S_n$-set $\Pairs^1_{k,n}$ defined in Theorem \ref{thm:RohiniQ1kn}. We prove:

\medskip

\noindent \textbf{Theorem \ref{thm:MainThmStrong}.} \textit{For all $k\ge0$ and $n\ge3$, $Q^2_{k,n}$ has a permutation basis, in equivariant bijection with the $S_n$-set $\Pairs^2_{k,n}$ defined in Definition \ref{Def:PairsDef}.}

\medskip

\noindent \textbf{Corollary \ref{cor:MainCor}.} \textit{The $S_n$-permutation representation $\Q(\Pairs^1_{k,n}\sqcup\Pairs^2_{k,n})$ is a subrepresentation of $H_{2k}(\Mbar_{0,n},\Q)$ for all $k$ and $n$.}

\medskip

By \eqref{eq:HDecomp}, $H_2(\Mbar_{0,n},\Q)=Q^1_{1,n}$ and $H_4(\Mbar_{0,n},\Q)\cong Q^1_{2,n}\oplus Q^2_{2,n}$. Thus Theorems \ref{thm:RohiniQ1kn} and \ref{thm:MainThmStrong} imply Theorem \ref{thm:MainThmWeak}. Note also that by Poincar\'e duality, for $n\le8$, these express $H_{2k}(\Mbar_{0,n},\Q)$ as a permutation representation for all $k$.

\subsection{Experimental evidence towards Question \ref{q:MainQ}} Let $V$ be a permutation representation of $G$ with permutation basis $B$. Then $B$ is a disjoint union of $G$-orbits, i.e. transitive $G$-sets. It is important to note that if $B_1$ and $B_2$ are two permutations bases for $V$, they need not be isomorphic as $G$-sets. However, the number of orbits of $B_1$ is equal to the dimension of the $G$-fixed subspace of $V$, thus equal to the numbers of orbits of $B_2$. Recall that every transitive $G$-set is equivariantly isomorphic to the $G$-set of left cosets of some subgroup of $G$. The above gives a computational strategy for checking whether a given $G$-representation $W$ is a permutation representation --- one enumerates the subgroups of $G$ and the associated characters, then checks whether the character of $W$ is a non-negative integer combination of those characters. This is doable in the GAP computer algebra system, at least through $n=10$.

%It is important to note that a permutation representation does not have a unique decomposition into transitive permutation representation -- this was alluded into in the paragraph after Question \ref{q:MainQ}.

J. Bergstr\"om and S. Minabe shared with us the characters of $H_{2k}(\Mbar_{0,n},\Q)$ for $n\le 20$. We now give some observations inferred from their data, from Theorem \ref{thm:MainThmStrong}, and from a GAP analysis as above. The smallest example where Theorems \ref{thm:RohiniQ1kn} and \ref{thm:MainThmStrong} do not answer Question \ref{q:MainQ} is $H_6(\Mbar_{0,9},\Q)$.
\begin{itemize}
    \item $Q^3_{3,9}$ is not a permutation representation. We computed the character of $Q^3_{3,9}$ by subtracting those of $Q^1_{3,9}$ and $Q^2_{3,9}$ (from Theorems \ref{thm:RohiniQ1kn} and \ref{thm:MainThmStrong}) from that of $H_6(\Mbar_{0,9},\Q)$ (from Bergstr\"om and Minabe's data).
    \item Despite the previous point, $H_6(\Mbar_{0,9},\Q)$ is a permutation representation. In fact, $H_6(\Mbar_{0,9},\Q)$ has two inequivalent decompositions into transitive permutation representations. It follows that in these two decompositions, the two permutation subrepresentations $\Q\Pairs^1_{3,9}$ and $\Q\Pairs^2_{3,9}$ \textit{cannot} both appear.
    \item $H_{2k}(\Mbar_{0,10},\Q)$ is a permutation representation for all $k$.
\end{itemize}

\subsection{Outline of the paper} In Section \ref{sec:Setup}, we recall fundamental results about the homology groups of $\Mbar_{0,n}$, and define $Q^r_{k,n}$, $\Pairs^1_{k,n}$, and $\Pairs^2_{k,n}$. In Section \ref{sec:ProofOutline} we outline the proof of Theorem \ref{thm:MainThmStrong}, and the complete proof is in Section \ref{sec:Proof}. In Section \ref{sec:conj}, we state a conjectural formula for the dimension of $Q^r_{k,n}$ in general.

\subsection*{Acknowledgements} Most of this work was conducted while
both authors were visiting the Max-Planck-Institut f\"{u}r Mathematik
in den Naturwissenschaften in Leipzig. We thank Bernd Stumrfels for
the invitations, as well as the entire institute for providing a
stimulating work environment.

We would like to thank J. Bergstr\"om and S. Minabe for generously
sharing data they had generated (based on their recursive formula) of
the characters of $H_{2k}(\Mbar_{0,n},\Q)$ for $n\le 20$ --- this was
crucial in allowing us to formulate and check conjectures. We are
grateful to Bernd Sturmfels, Renzo Cavalieri, Maria Monks Gillespie
and Melody Chan for useful conversations.

The first author was supported by NSF grants DMS-0943832 and DMS-1703308, and by a postdoctoral position at Brown University. The second author was supported by NSF grants  DMS-0602191 and DMS-1645877, and by a postdoctoral position at Northeastern University.

\section{Background and notation}\label{sec:Setup}
% Should mention:
% \begin{itemize}
% \item Keel generators and relations
% \item Rohini filtration
% \item Rohini thesis result about $Q^{1}_{k,n}$.
% \item Stronger version of theorem, about $Q^{2}_{k,n}$.
% \item Negative result about $Q^{3}_{3,9}.$
% \end{itemize}
\begin{Def}
  Let $\mathbf{S}_{k,n}$ denote the set of stable trees with $n-2-k$ vertices, $n-3-k$ edges, and $n$
  marked half-edges (labeled by the set $\{1,\ldots,n\})$. (A tree is \textit{stable} if each vertex has valence at least 3.) $\mathbf{S}_{k,n}$ is in canonical bijection with the set of (closed) $k$-dimensional boundary strata in $\Mbar_{0,n}.$
\end{Def}
%\begin{notation}
We will abuse notation, and refer to an element of $\mathbf{S}_{k,n}$
interchangeably as a marked tree, as a closed $k$-dimensional
subvariety of $\Mbar_{0,n}$, and as a cycle class in
$H_{2k}(\Mbar_{0,n})$ or various quotients thereof. For example,
%\end{notation}
it is a well-known result of Keel \cite{Keel1992} that $H_{2k}(\Mbar_{0,n},\Q)$ is
spanned by
$\mathbf{S}_{k,n}.$ Kontsevich and Manin \cite{KontsevichManin1994} gave a spanning set for the kernel of the map $\Q\mathbf{S}_{k,n}\to H_{2k}(\Mbar_{0,n},\Q)$. Let $\sigma\in\mathbf{S}_{k+1,n},$ let $v$ be a vertex of $\sigma$ with valence at least 4, and let $A,B,C,D$ be four edges or half-edges incident to $v$. Let $\mathbf{T}$ denote the set of edges and half-edges incident to $v$, other than $A,B,C,D$. We define a relation
\begin{align}\label{eq:KMRelation}
\mathcal{R}=\mathcal{R}(\sigma,v,A,B,C,D)=\sum_{\mathbf{U}_1\sqcup\mathbf{U}_2=\mathbf{T}}(AB\mathbf{U}_2|CD\mathbf{U}_2)-\sum_{\mathbf{U}_1\sqcup\mathbf{U}_2=\mathbf{T}}(AC\mathbf{U}_2|BD\mathbf{U}_2)\in\Q\mathbf{S}_{k,n}.
\end{align}
We now explain this notation, which we will continue to use. Here $(AB\mathbf{U}_1|CD\mathbf{U}_2)\in\mathbf{S}_{k,n}$ is the tree obtained from $\sigma$ by replacing $v$ with two vertices $v',v''$ connected by an edge $e'$, with the edges or half-edges $\{A,B\}\cup\mathbf{U}_1$ incident to $v'$ and the edges or half-edges $\{C,D\}\cup\mathbf{U}_2$ incident to $v''.$ Note that contracting $e'$ yields $\sigma$ again. Kontsevich and Manin showed that the relations obtained this way (varying over all choices of $\sigma$, $v$, $A,B,C,D$) span the kernel of the map $\Q\mathbf{S}_{k,n}\to H_{2k}(\Mbar_{0,n},\Q)$.

Note that for any
$\sigma\in\mathbf{S}_{k,n},$ we have
\begin{align}\label{eq:TreeToPartition}
\sum_{v}(\val(v)-3)=k,
\end{align}
where the sum is over vertices of $\sigma$.
Thus to any $\sigma\in\mathbf{S}_{k,n}$ we may associate a partition
$\lambda_\sigma$ of $k$, namely the set of nonzero summands on the
right-hand side of \eqref{eq:TreeToPartition}. Since $\abs{V(\sigma)}=n-2-k$, $\lambda_\sigma$ has at most $n-2-k$ parts. (As a partition of $k$, $\lambda_\sigma$ also has at most $k$ parts.)
\begin{Def}\label{Def:RamadasFiltration}\cite{Ramadas2015}
  For $r\in\{1,\ldots,\min\{k,n-2-k\}\},$ let
  $\mathbf{S}_{k,n}^{\ge r}\subseteq\mathbf{S}_{k,n}$ be the subset
  consisting of those $\sigma\in\mathbf{S}_{k,n}$ such that
  $\lambda_\sigma$ has at least $r$ parts. This defines an $S_n$-invariant
  filtration
  $$\mathbf{S}_{k,n}=\mathbf{S}_{k,n}^{\ge1}\supseteq\mathbf{S}_{k,n}^{\ge2}\supseteq\cdots\supseteq\mathbf{S}_{k,n}^{\ge
    k}.$$ Since $H_{2k}(\Mbar_{0,n},\Q)$ is spanned by
  $\mathbf{S}_{k,n},$ this also defines an $S_n$-invariant filtration
  \begin{align}\label{eq:RamadasFiltration}
    H_{2k}(\Mbar_{0,n},\Q)=Q_{k,n}^{\ge1}\supseteq
    Q_{k,n}^{\ge2}\supseteq\cdots\supseteq Q_{k,n}^{\ge k},
  \end{align}
  where
  $$Q^{\ge r}_{k,n}=\Span(\mathbf{S}_{k,n}^{\ge r})\subseteq
  H_{2k}(\Mbar_{0,n},\Q).$$ Finally, let $\mathbf{S}_{k,n}^r=\mathbf{S}_{k,n}^{\ge
    r}\setminus\mathbf{S}_{k,n}^{\ge r+1}$ and 
  $Q^r_{k,n}=Q_{k,n}^{\ge r}/Q^{\ge r+1}_{k,n}.$ Note that
  $\mathbf{S}_{k,n}^r$ spans $Q_{k,n}^r$, and that as $S_n$-representations, $$H_{2k}(\Mbar_{0,n},\Q)\cong\bigoplus_{r=1}^{\min\{k,n-2-k\}}Q^r_{k,n}.$$
\end{Def}
In \cite{Ramadas2015}, it is shown that the filtration \eqref{eq:RamadasFiltration} is preserved by pushforwards along tautological morphisms of moduli spaces $\Mbar_{0,n}$. In \cite{RamadasThesis}, $Q^1_{k,n}$ is computed explicitly:
\begin{thm}[\cite{RamadasThesis}]\label{thm:RohiniQ1kn}
  Let
  $\Pairs^1_{k,n}=\{A\subseteq\{1,\ldots,n\}:k+3\le\abs{A}\le
  n,\thickspace\abs{A}\equiv k+3\mod2\},$ as an $S_n$-set. Then
  $Q^1_{k,n}\cong\Q\Pairs^1_{k,n}$ as $S_n$-representations.
\end{thm}
Since $H_2(\Mbar_{0,n},\Q)$ consists of the single graded piece
$Q^1_{1,n}$, this implies:
\begin{cor}[\cite{RamadasThesis}]\label{cor:RohiniQ1kn}
  For $n\ge3$,
  $H_{2}(\Mbar_{0,n},\Q)\cong\Pairs^1_{1,n}\cong\{A\subseteq\{1,\ldots,n\}:4\le\abs{A}\le
  n,\thickspace\abs{A}\text{ even}\}$ as $S_n$-representations.
  \end{cor}

\begin{Def}\label{Def:PairsDef}
  Let $\Pairs_{k,n}^2$ denote the set of unordered pairs
  $\{(P_1,\alpha_1),(P_2,\alpha_2)\}$, where:
  \begin{itemize}
  \item $P_1$ and $P_2$ are disjoint subsets of $\{1,\ldots,n\}$, and
  \item $\alpha_1$ and $\alpha_2$ are positive integers whose sum is
    $k$, and
  \item $P_i\ge\alpha_i+2$ for $i\in\{1,2\}$.
  \end{itemize}
\end{Def}
Observe that Definition \ref{Def:PairsDef} is consistent with the the
notation $\Pairs^2_{2,n}$ from Theorem \ref{thm:MainThmWeak}.
\begin{Def}\label{Def:WPDef}
  We define an $S_n$-equivariant map
  $W_n:\mathbf{S}^2_{k,n}\to\Pairs^2_{k,n}$ as follows. For
  $\sigma\in\mathbf{S}^2_{k,n}$, let $v_1$ and $v_2$ be the two
  vertices of $\sigma$ with valence at least 4.  There is a unique
  edge $e_1$ incident to $v_1$ that is contained in any path from
  $v_1$ to $v_2$, and similarly an edge incident $e_2$ incident to
  $v_2$. Cutting the two edges $e_1$ and $e_2$ (into unmarked half-edges) yields three connected
  components $\sigma_1,$ $\sigma'$, and $\sigma_2.$ (Note: the graphs $\sigma_1$ and $\sigma_2$ have a single ``central'' vertex, with trivalent subtrees incident to it. The graph $\sigma'$ is itself a trivalent subtree.)

  Let $\alpha_1=\val(v_1)-3$, and let $P_1\subseteq\{1,\ldots,n\}$ be the set
  of marked half-edges in $\sigma_1.$ Similarly let $\alpha_2=\val(v_2)-3$, and let $P_2\subseteq\{1,\ldots,n\}$ be the set
  of marked half-edges in $\sigma_2.$ %  be the set and
  % $\alpha_2=\val(v_2)-3.$ Let $P_1\subseteq\{1,\ldots,n\}$ be the set
  % of half-edges in the connected component of $v_1$ when $e_1$ is
  % cut. Similarly we define $e_2$ and
  % $P_2$.
  % $\sigma$ has two vertices $v_1$ and $v_2$ with moduli
  % dimensions $\alpha_1$ and $\alpha_2$. Let:
  % \begin{align*}
  %   P_1:&=\{1,\ldots,n\}\setminus\SP_{v_1}(v_2)\\
  %   P_2:&=\{1,\ldots,n\}\setminus\SP_{v_2}(v_1).
  % \end{align*}
  Then we define $W_n(\sigma):=\{(P_1,\alpha_1),(P_2,\alpha_2)\}.$
  Observe that $W_n(\sigma)\in\Pairs^2_{k,n},$ since:
  \begin{itemize}
  \item $P_1$ and $P_2$ are disjoint (by definition),
  \item $\alpha_1+\alpha_2=k$ (this follows from the fact that
    $\sigma$ is $k$-dimensional), and
  \item $\abs{P_i}\ge\alpha_i+2$ for $i\in\{1,2\}$. (This is equivalent
    to $\abs{P_i}\ge\val(v_i)-1,$ which follows from the fact that by stability, every edge incident to $v_i$, except $e_1$, is on a nonrepeating path from $v_i$ to at least one marked half-edge in $\sigma_1$.)
  \end{itemize}
\end{Def}

\begin{thm}\label{thm:MainThmStrong}
  For any $n>3$ and $2\le k\le n-4$, $Q^{2}_{k,n}$ is isomorphic (as
  an $S_n$-module) to $\Q\Pairs^2_{k,n}$.
\end{thm}
Together with Theorem \ref{thm:RohiniQ1kn} and the fact that
$H_4(\Mbar_{0,n},\Q)$ consists of the two graded pieces $Q^1_{2,n}$
and $Q^2_{2,n},$ Theorem \ref{thm:MainThmStrong} implies:
\begin{thm}\label{thm:MainThmWeak}
  For $n\ge3,$
  $H_4(\Mbar_{0,n},\Q)\cong\Q\Pairs^1_{2,n}\oplus\Q\Pairs^2_{2,n}$ as $S_n$-representations.
\end{thm}

\begin{cor}\label{cor:MainCor}
    For $n\ge3$ and $k\ge0,$ $\Q\Pairs^1_{k,n}\oplus\Q\Pairs^2_{k,n}$ is a subrepresentation of $H_{2k}(\Mbar_{0,n},\Q)$.
\end{cor}
% Theorem ?? from \cite{RamadasThesis} states that
% $Q^1_{2,n}\cong\Q\Pairs^1_{2,n}$. When $k=2$, the filtration
% \eqref{eq:RamadasFiltration} has only two graded pieces, so Theorem
% \ref{thm:MainThmStrong} implies Theorem \ref{thm:MainThmWeak}.

\section{Proof of Theorem \ref{thm:MainThmStrong}}
\subsection{Outline of Proof}\label{sec:ProofOutline}
We first introduce (Definition
\ref{Def:Filtrations}) filtrations
\begin{align*}
  \mathbf{S}^2_{k,n}&=(\mathbf{S}^2_{k,n})^{\ge0}\supseteq
                  (\mathbf{S}^2_{k,n})^{\ge1}\supseteq\cdots\supseteq(\mathbf{S}^2_{k,n})^{\ge
                  n-k-4}\\
  \Pairs^2_{k,n}&=(\Pairs^2_{k,n})^{\ge0}\supseteq
                  (\Pairs^2_{k,n})^{\ge1}\supseteq\cdots\supseteq(\Pairs^2_{k,n})^{\ge
                  n-k-4}\\
  Q^{2}_{k,n}&=(Q^{2}_{k,n})^{\ge0}\supseteq
               (Q^{2}_{k,n})^{\ge1}\supseteq\cdots\supseteq (Q^{2}_{k,n})^{\ge
               n-k-4},
\end{align*}
with graded pieces
\begin{align*}
  (\mathbf{S}^2_{k,n})^{b}&=(\mathbf{S}^2_{k,n})^{\ge
                            b}\setminus(\mathbf{S}^2_{k,n})^{\ge b+1}\\
  (\Pairs^2_{k,n})^{b}&=(\Pairs^2_{k,n})^{\ge
                            b}\setminus(\Pairs^2_{k,n})^{\ge b+1}\\
  (Q^2_{k,n})^{b}&=(Q^2_{k,n})^{\ge
                            b}/(Q^2_{k,n})^{\ge b+1}.
\end{align*}
Then $W_n$ will induce a map
$W_{n,b}:(\mathbf{S}^2_{k,n})^{b}\to(\Pairs^2_{k,n})^{b}$. We will
have surjective maps:
\begin{center}
  \begin{tikzcd}
    &\Q(\mathbf{S}^2_{k,n})^{b}\arrow[dl,"\rho_{n,b}",swap]\arrow[dr,"W_{n,b}"]&\\
    (Q^{2}_{k,n})^{b}&&\Q(\Pairs^2_{k,n})^{b}
  \end{tikzcd}
\end{center}
where we abuse notation and use $W_{n,b}$ to refer to the map induced on free $\Q$-vector spaces by
$W_{n,b}$. We show (Lemma \ref{lem:QToP}) that $W_{n,b}$ factors
equivariantly through $\rho_{n,b}$, and (Lemma \ref{lem:PToQ}) that
$\rho_{n,b}$ factors equivariantly through $W_{n,b}$; the two
resulting maps $(Q^{2}_{k,n})^{b}\to\Q(\Pairs^2_{k,n})^{b}$ and
$\Q(\Pairs^{2}_{k,n})^{b}\to(Q^2_{k,n})^{b}$ must therefore be
inverses, by a straightforward diagram chase.

\subsection{Proofs of Lemmas \ref{lem:QToP} and \ref{lem:PToQ}}\label{sec:Proof}

\begin{Def}\label{Def:Filtrations}
  For $\sigma\in\mathbf{S}^2_{k,n}$, let $\sigma_1$, $\sigma_2$, and
  $\sigma'$ be as in Definition \ref{Def:WPDef}. Then for
  $0\le b\le n-k-4$, we let $(\mathbf{S}^2_{k,n})^{\ge b}$ denote the
  subset consisting of those trees $\sigma$ such that $\sigma'$ has
  at least $b$ marked half-edges. Let
  \begin{align*}
    (Q^{2}_{k,n})^{\ge b}&=\Span((\mathbf{S}^2_{k,n})^{\ge b})\subseteq
                           Q^2_{k,n}\\
    (Q^{2}_{k,n})^{b}&=(Q^{2}_{k,n})^{\ge b}/(Q^{2}_{k,n})^{\ge b+1}.
  \end{align*}
  Let
  \begin{align*}
    (\Pairs^2_{k,n})^{\ge b}&=W_{n}((\mathbf{S}^2_{k,n})^{\ge b})\\
    (\Pairs^2_{k,n})^{b}&=(\Pairs^2_{k,n})^{\ge b}\setminus(\Pairs^2_{k,n})^{\ge b+1}.
  \end{align*}
\end{Def}
We write $\rho_{n,b}$ for the natural map
$\Q(\mathbf{S}^2_{k,n})^{b}\to(Q^2_{k,n})^{b}$. Observe that
$\rho_{n,b}$ is surjective, and that
$W_{n}|_{(\mathbf{S}^2_{k,n})^{b}}\subseteq(\Pairs^2_{k,n})^{b}$ (so
$W_{n,b}$ is well-defined). We now prove:
% \begin{Def}\label{Def:PairsFiltration}
%   % For $0\le b\le n-k-4$,
%   % let
%   % $$(\Pairs^2_{k,n})^{\ge b}=\{\{(P_1,\alpha_1),(P_2,\alpha_2)\}\in\Pairs^2_{k,n}:\abs{(P_1\cup
%   %   P_2)^C}\ge b\}.$$ Let
%   % $(\Pairs^2_{k,n})^{b}=(\Pairs^2_{k,n})^{\ge b}\setminus(\Pairs^2_{k,n})^{\ge b+1}.$
% \end{Def}
% \begin{Def}\label{Def:Q2knFiltration}
%   Recall that $Q_{k,n}^2$ is spanned by $\mathbf{S}^2_{k,n},$ which
%   consists of marked trees $\sigma$ such that exactly two vertices
%   $v_1$ and $v_2$ have valence greater than $3.$ Let
%   $(Q^{2}_{k,n})^{\ge b}$ be the subspace of $Q^{2}_{k,n}$ spanned by
%   trees $\sigma\in\mathbf{S}^2_{k,n}$ such that the vertices strictly
%   between $v_1$ and $v_2$ contain, in total, at least $b$ marked
%   half-edges. Let
%   $(Q^{2}_{k,n})^{b}=(Q^{2}_{k,n})^{\ge b}/(Q^{2}_{k,n})^{\ge b+1}.$
%   % For convenience, we also define
%   % \begin{align*}
%   %   (\Pairs^2_{k,n})^{\ep\le c}:&=(\Pairs^2_{k,n})^{\pd\ge n-k-4-c}\\
%   %   (\Pairs^2_{k,n})^{\ep=c}:&=(\Pairs^2_{k,n})^{\ep\le c}/(\Pairs^2_{k,n})^{\ep\le c-1}\\
%   %   Q^{\ep\le c}:&=Q^{\pd\ge n-k-4-c}\\
%   %   Q^{\ep=c}:&=Q^{\ep\le c}/Q^{\ep\le c-1}.
%   % \end{align*}
% \end{Def}

% To prove Theorem \ref{thm:MainThmStrong}, we first show (Proposition
% \ref{prop:PairsBIsom}) that there is an $S_n$-equivariant isomorphism
% $(Q^{2}_{k,n})^{b}\to(\Pairs^2_{k,n})^{b}$ for all $b,k,n$.

% In particular, this gives an identification of
% $(Q^{2}_{k,n})^{\pd=n-k-4}$ with $\Q\langle(\Pairs^2_{k,n})^{\pd=n-k-4}\rangle$,
% since $(Q^{2}_{k,n})^{\pd=n-k-3}=0.$

\begin{lem}\label{lem:QToP}
  Fix nonnegative integers $n\ge4,$ $k\le n-3$, and $b\le n-k-4$. The
  map $W_{n,b}:\Q(\mathbf{S}^2_{k,n})^{b}\to\Q(\Pairs^2_{k,n})^{b}$
  descends equivariantly to an $S_n$-equivariant map
  $\bar{W_{n,b}}:(Q^{2}_{k,n})^b\to\Q(\Pairs^2_{k,n})^{b}$. % (Equivalently, $W$ descends to an $S_n$-equivariant map
  % $(Q^{2}_{k,n})^{\ep=c}\to(\Pairs^2_{k,n})^{\ep=c}$ for all $c,k,n$.)
\end{lem}
% \begin{lem}\label{lem:Pairs0Isom}
%   For all $k$ and $n$, $W$ descends equivariantly to a map
%   $\bar{W}:(Q^{2}_{k,n})^{0}\to(\Pairs^2_{k,n})^{0},$ sending the
%   class $\alpha$ of a boundary stratum to
%   $W(\alpha)$.% (Equivalently, this is an $S_n$-equivariant
%   % isomorphism
%   % $(Q^{2}_{k,n})^{\ep=n-k-4}\to(\Pairs^2_{k,n})^{\ep=n-k-4}.$)
% \end{lem}
% \begin{proof}
  
% \end{proof}
The proof is by induction on $b$. The base case $b=0$ (for all $n$) is
Proposition \ref{prop:Pairs0Isom}. The inductive step is Proposition
\ref{prop:PairsBIsom}.
\begin{prop}\label{prop:Pairs0Isom}
  $W_{n,0}:\Q(\mathbf{S}^2_{k,n})^{0}\to\Q(\Pairs^2_{k,n})^{0}$
  descends equivariantly to an $S_n$-equivariant map
  $\bar{W_{n,0}}:(Q^{2}_{k,n})^0\to\Q(\Pairs^2_{k,n})^{0}$.
\end{prop}
The proof is complicated, so we first give a discussion. Note that
  $$(Q^{2}_{k,n})^0=\frac{\Q\mathbf{S}^{\ge2}_{k,n}}{R^{\ge2}+\Q\mathbf{S}^{\ge3}_{k,n}+\Q(\mathbf{S}^{2}_{k,n})^{\ge1}}.$$
  Here
  $R^{\ge2}:=R\cap\Q\mathbf{S}^{\ge2}_{k,n}\subseteq\Q\mathbf{S}^{\ge2}_{k,n}$,
  where $R\subseteq\Q\mathbf{S}_{k,n}$ is the kernel of the map $\Q\mathbf{S}_{k,n}\to H_{2k}(\Mbar_{0,n},\Q)$. It is difficult to work directly with $R^{\ge2}$, due to the fact
  that the filtration $(\mathbf{S}^{\ge2}_{k,n})^{\bullet}$ does not interact
  in a predictable way with the Kontsevich-Manin relations \eqref{eq:KMRelation}. Our strategy is as follows. We will find a map
  $\tWPQn:\Q\mathbf{S}_{k,n}\to\Q\Pairs^2_{k,n}$ such
  that:
  \begin{enumerate}[label=(\roman*)]
  \item The restriction of $\tWPQn$ to
    $\Q(\mathbf{S}^{2}_{k,n})^0$ is $W_{n,0}$,\label{i1}
  \item The image of the restriction of $\tWPQn$ to
    $\Q(\mathbf{S}^{2}_{k,n})^{\ge1}$ lies in
    $\Q(\Pairs^2_{k,n})^{\ge1}$,\label{i2}
  \item The restriction of $\tWPQn$ to
    $\Q\mathbf{S}^{\ge3}_{k,n}$ is zero, and\label{i3}
  \item For any Kontsevich-Manin relation $\mathcal{R}\in\Q\mathbf{S}_{k,n},$
    $\tWPQn({\mathcal{R}})\in\Q(\Pairs^2_{k,n})^{\ge1}$.\label{i4}
  \end{enumerate}
  Conditions \ref{i2}, \ref{i3}, and \ref{i4} imply that
  $\tWPQn$ descends to a map
  \begin{align*}
    \frac{\Q\mathbf{S}_{k,n}}{R+\Q\mathbf{S}^{\ge3}_{k,n}+\Q(\mathbf{S}^{2}_{k,n})^{\ge1}}\to\Q(\Pairs^2_{k,n})^0.
  \end{align*}
  Restricting to
  $\Q\mathbf{S}^{\ge2}_{k,n}\subseteq\Q\mathbf{S}_{k,n}$, Condition
  \ref{i1} gives a map $\bar{W_{n,0}}:(Q^{2}_{k,n})^0\to\Q(\Pairs^2_{k,n})^{0}$
  descended from $W_{n,0}$.
  \begin{rem}
    The existence of a map satisfying conditions \ref{i1} -- \ref{i4} is quite mysterious to us; we found $\tWPQn$ by computer experimentation. The fact that the definition of $\tWPQn$ (see \eqref{eq:WPTilde1} below) is quite complicated makes us wonder if there is some simpler, more conceptual description that might simplfy the following proof.
  \end{rem}
\begin{proof}[Proof of Proposition \ref{prop:Pairs0Isom}]
  Let $\sigma\in\mathbf{S}^1_{k,n}$, and let $v$ be the unique vertex of
  $\sigma$ with valence $\ge4.$ Then $\val(v)=k+3.$ We define a set partition $\Pi\vdash\{1,\ldots,n\}$ induced by $\sigma$, where $i_1$ and $i_2$ are in the same part if and only if the corresponding marked half-edges are in the same connected component of the complement of $v$ in $\sigma.$ Note that $\Pi$ has $k+3$ parts. 
  
  Suppose $\Gamma=\{(P_1,\alpha_1),(P_2,\alpha_2)\}\in(\Pairs^2_{k,n})^0$ is such that $P_1$ is a union of $s_1$ parts of $\Pi$ (so necessarily $P_2$ is a union of the remaining $s_2:=k+3-s_1$ parts of $\Pi$). 
%   Since $\alpha_1+\alpha_2=k,$ we have $(s_1-\alpha_1)+(s_2-\alpha_2)=(k+3)-k=3.$ In particular, $s_1-\alpha_1\ne s_2-\alpha_2$. 
  Let $e_\Pi(\Gamma)=\min\{s_1-\alpha_1,s_2-\alpha_2\}.$ We define:
  \begin{align}\label{eq:WPTilde1}
    \tWPQn(\sigma):=\sum_{\substack{\Gamma\in(\Pairs^2_{k,n})^0\\\Gamma=\{(P_1,\alpha_1),(P_2,\alpha_2)\}\\\text{$P_1$ is a union of parts of $\Pi$.}}}\frac{(-1)^{e_\Pi(\Gamma)}}{2}\Gamma.
  \end{align}
%   \begin{align}\label{eq:WPTilde1}
%     \tWPQn(\sigma):=\sum_{j=0}^{k-1}\sum_{\substack{\Omega\subseteq
%     \Pi\\2-j\le\abs{\Omega}\le k-j\\\abs{\Omega}+1+j\le\sum_{S\in \Omega}\abs{S}}}\frac{(-1)^{j+1}}{2}% \left(
%     % \begin{array}{c|c}
%     %   \bigcup_{a\in T}a&\bigcup_{a\in\Pi\setminus T}a\\\hline
%     %   \abs{T}-1+j&k+3-\abs{T}-2-j
%     % \end{array}
%     %                \right)
%                   \left\{\textstyle\left(\bigcup_{S\in \Omega}S,\abs{\Omega}-1+j\right),\left(\bigcup_{S\in \Omega^C}S,k-\abs{\Omega}+1-j\right)\right\}.
%   \end{align}
  This defines $\tWPQn$ on
  $\mathbf{S}^1_{k,n}$. We also define
  \begin{align}\label{eq:WPTilde23}
    \tWPQn(\sigma)=
    \begin{cases}
      W_{n}(\sigma)&\sigma\in\mathbf{S}^2_{k,n}\\
      0&\sigma\in\mathbf{S}^{\ge3}_{k,n}.
    \end{cases}
  \end{align}
  Extending by linearity, these collectively define a map
  $\tWPQn:\Q\mathbf{S}_{k,n}\to\Q\Pairs^2_{k,n},$ which clearly satisfies conditions \ref{i1}, \ref{i2}, and \ref{i3} above.

  Let $\mathcal{R}=\mathcal{R}(\sigma,v,A,B,C,D)\in\Q\mathbf{S}_{k,n}$ be a Kontsevich-Manin
  relation as in \eqref{eq:KMRelation}. We must show $\tWPQn({\mathcal{R}})\in\Q(\Pairs^2_{k,n})^{\ge1}$.
  There are several cases, which we treat separately -- Case VI is the
  hardest by far:
  \begin{enumerate}[label=\textbf{\Roman*.}]
  \item $\sigma$ has $\ge4$ vertices with valence $\ge4$.
  \item $\sigma$ has 3 vertices $v,v',v''$ with valence $\ge4$, and
    $\val(v)\ge5$.
  \item $\sigma$ has 3 vertices $v,v',v''$ with valence $\ge4$,
    and $\val(v)=4$.
  \item $\sigma$ has exactly 2 vertices $v,v'$ with valence $\ge4$, and
    $\val(v)>4$.
  \item $\sigma$ has exactly 2 vertices $v,v'$ with valence $\ge4$, and
    $\val(v)=4$.
  \item $v$ is the unique vertex of $\sigma$ with $\val(v)\ge4$.
  \end{enumerate}

  \medskip

  \noindent\textbf{Case I.} In this case, every term of ${\mathcal{R}}$ is in
  $\mathbf{S}^{\ge3}_{k,n}$, so $\tWPQn({\mathcal{R}})=0.$

  \medskip

  \noindent\textbf{Case II.} Again, every term of ${\mathcal{R}}$ is in
  $\mathbf{S}^{\ge3}_{k,n}$, so $\tWPQn({\mathcal{R}})=0.$
  
  \medskip
  
  \noindent\textbf{Case III.} % There are three subcases (FIX
  % THE FOLLOWING):
  % \begin{enumerate}
  % \item $v_r$ is on the path between the other two vertices with
  %   moduli, and $A_r$ and $D_r$ are both (or neither are) nodes on that
  %   path.
  % \item $v_r$ is on the path between the other two vertices with
  %   moduli, and exactly one of $A_r$ and $D_r$ are nodes on that path.
  % \item $v_r$ is not on the path between the other two vertices with
  %   moduli.
  % \end{enumerate}
  In this case, ${\mathcal{R}}$ has exactly two terms, with opposite signs. These
  two trees (which both have exactly two vertices with
  valence $\ge4$) differ only in the rearrangement of trivalent
  subtrees. Thus $\tWPQn$ does not distinguish
  between them, i.e. $\tWPQn({\mathcal{R}})=0.$
  
  % Note: In fact, in the first case, we would have
  % $r\in V_{2,k,n}^{\ge2}$ anyway, so automatically we have
  % $\tilde{W}(r)\in(\Pairs^2_{k,n})^{\ge2}\subseteq(\Pairs^2_{k,n})^{\ge1}$.

  \medskip

  \noindent\textbf{Case IV.} All but four terms of ${\mathcal{R}}$ are in
  $\mathbf{S}^{\ge3}_{k,n},$ so we may ignore them. If $v$ and $v'$
  are not adjacent, then these four terms are all in
  $(\mathbf{S}^{2}_{k,n})^{\ge1},$ hence
  $\tWPQn({\mathcal{R}})\in(\Pairs^2_{k,n})^{\ge1}.$ If $v$
  and $v'$ are adjacent, connected by an edge $e$, there are two
  subcases:
\begin{enumerate}
\item One of $A,B,C,D$ is $e$ (without loss of generality,
  $D=e$), or\label{item:IV1}
\item None of $A,B,C,D$ is $e$.\label{item:IV2}
\end{enumerate}
In subcase \ref{item:IV1}, let $\mathbf{T}$ be the set of edges or half-edges incident
to $v$ other than $A,B,C,D$. In the notation of Section \ref{sec:Setup}, the four remaining
terms of ${\mathcal{R}}$ are:
\begin{align*}
  (AB\mathbf{T}|Ce)+(AB|Ce\mathbf{T})-(AC\mathbf{T}|Be)-(AC|Be\mathbf{T}).
\end{align*}

Note that by definition
$\tWPQn(AB|Ce\mathbf{T})=\tWPQn(AC|Be\mathbf{T})$. The first and third terms are sent to
$(\Pairs^2_{k,n})^{\ge1}.$ Thus
$\tWPQn({\mathcal{R}})\in(\Pairs^2_{k,n})^{\ge1}.$

In subcase \ref{item:IV2}, let $\mathbf{T}$ be the set of edges or half-edges
incident to $v$ other than $A,B,C,D,e$. The four remaining terms of ${\mathcal{R}}$ are:
\begin{align*}
  (ABe\mathbf{T}|CD)+(AB|CDe\mathbf{T})-(ACe\mathbf{T}|BD)-(AC|BDe\mathbf{T}).
\end{align*}
% \begin{align*}
%   (A_rB_r\mathbf{T}_r|C_rD_re)+(A_rB_r|C_rD_re\mathbf{T}_r)-(A_rC_r\mathbf{T}_r|B_rD_re)-(A_rC_r|B_rD_re\mathbf{T}_r).
% \end{align*}
% The first two terms are in $V_{3,k,n}$, and the last two terms cancel
% after applying $\tilde{W}$.
% If $T$ is empty, then diagrammatically, $r$ is:
% \begin{align*}
%   (A_rB_r|C_rD_rN_r)-(A_rC_r|B_rD_rN_r).
% \end{align*}
% These cancel after applying $\tilde{W}$.
Observe that
$$\tWPQn(ABe\mathbf{T}|CD)=\tWPQn(AB|CDe\mathbf{T})=\tWPQn(ACe\mathbf{T}|BD)=\tWPQn(AC|BDe\mathbf{T}).$$
Thus $\tWPQn({\mathcal{R}})=0.$

\medskip

\noindent\textbf{Case V.} ${\mathcal{R}}$ has two
terms with opposite sign, each of which is a tree with
exactly one vertex with valence $\ge4.$ They induce the same partition
$\Pi\vdash\{1,\ldots,n\}$ (as in \eqref{eq:WPTilde1}), hence they
cancel after applying $\tWPQn$.

\medskip

\noindent\textbf{Case VI.} We have $\val(v)=k+4$. We write $\mathbf{T}$ for the set
of edges and half-edges incident to $v$ other than
$A,B,C,D$. Note that
$\abs{\mathbf{T}}=k$.
  
Recall that
\begin{align*}
  \mathcal{R}=\sum_{\mathbf{U}_1\sqcup \mathbf{U}_2=\mathbf{T}}(AB\mathbf{U}_1|CD\mathbf{U}_2)-\sum_{\mathbf{U}_1\sqcup
  \mathbf{U}_2=\mathbf{T}}(AC\mathbf{U}_1|BD\mathbf{U}_2)\in\Q\mathbf{S}_{k,n}.
\end{align*}
Note that terms where $\mathbf{U}_1=\emptyset$ or
$\mathbf{U}=\emptyset$ are in $\mathbf{S}^1_{k,n}$, and terms where
$\mathbf{U}_1,\mathbf{U}_2\ne\emptyset$ are in $\mathbf{S}^2_{k,n}.$ We will apply
% \eqref{eq:WPTilde1} or \eqref{eq:WPTilde23} accordingly. We note
% two symmetries of ${\mathcal{R}}$ that will be useful:

% \medskip

% \noindent\textit{Symmetry 1.} The operation of transposing $B$
% with $C$ sends
% ${\mathcal{R}}\mapsto-{\mathcal{R}}$. Similarly for transposing
% $A$ with $D$.

% \smallskip

% Applying Symmetry 1 twice, we have:

% \smallskip
  
% \noindent\textit{Symmetry 2.} The operation of transposing $A$ with $D$ and
% $B$ with $C$ preserves ${\mathcal{R}}$.% (This operation permutes the
% terms of $r$, sending the term corresponding to
% $\mathbf{U}_1\sqcup\mathbf{U}_2$ to that corresponding to
% $\mathbf{U}_2\sqcup\mathbf{U}_1$.)

% (This switches the term
% $D(AB\mathbf{U}_1|CD\mathbf{U}_2)$ in the first sum with
% $D(AC\mathbf{U}_1|BD\mathbf{U}_2)$ in the second sum. These two terms
% differ by a sign.)

\medskip

We also introduce an abuse of notation that will now be convenient. By definition, $A,B,C,D,$ and the elements of $\mathbf{T}$ denote edges or half-edges of $\sigma$ incident to $v.$ We use the same symbols to refer to subsets of $\{1,\ldots,n\}$, where e.g. $A$ is the set of all $i\in\{1,\ldots,n\}$ such that the path from $v$ to the $i$th marked half-edge contains the edge $A$. (If $A$ is itself the $i$th marked half-edge, then we write $A=\{i\}$.)

Let $\Gamma=\{(P_1,\alpha_1),(P_2,\alpha_2)\}\in(\Pairs^2_{k,n})^0$. Note that by definition of $(\Pairs^2_{k,n})^0,$ we
have $P_1\cup P_2=\{1,\ldots,n\}.$ We need to show that the coefficient of $\Gamma$ in $\tWPQn({\mathcal{R}})$ is zero. We do this argument in cases again; $\Gamma$ must be of one of the following types:

\medskip

\noindent\textit{Type 0.} At least one of the sets $A,B,C,D\subseteq\{1,\ldots,n\}$, or a part of
$\mathbf{T}$, has nonempty intersection with both $P_1$ and $P_2$.

\medskip
  
\noindent\textit{Type 1.} Two of $A,B,C,D$ are subsets of $P_1$, and
the other two are subsets of $P_2$, and no part of $\mathbf{T}$ has nonempty intersection with both $P_1$ and $P_2$.

\smallskip

\textit{Type 1.1} $A,B\subseteq P_1$ and $C,D\subseteq P_2$,
% and no part of $\mathbf{T}$ is split up between $P_1$ and $P_2$ (or
% vice versa, since $P_1$ and $P_2$ are indistinguishable),
and no part of $\mathbf{T}$ has nonempty intersection with both $P_1$ and $P_2$.

\smallskip
  
\textit{Type 1.2.} $A,C\subseteq P_1$ and $B,D\subseteq P_2$,
and no part of $\mathbf{T}$ has nonempty intersection with both $P_1$ and $P_2$.

\smallskip

\textit{Type 1.3.} $A,D\subseteq P_1$ and $B,C\subseteq P_2$,
and no part of $\mathbf{T}$ has nonempty intersection with both $P_1$ and $P_2$.

\medskip

\noindent\textit{Type 2.} Three of $A,B,C,D$ are subsets of $P_1$, and
the remaining one is a subset of $P_2,$ and no part of $\mathbf{T}$ has nonempty intersection with both $P_1$ and $P_2$.

\smallskip

\textit{Type 2.1} $A,B,C\subseteq P_1$ and $D\subseteq P_2$,
and no part of $\mathbf{T}$ has nonempty intersection with both $P_1$ and $P_2$.

\smallskip
  
\textit{Type 2.2.} $A,C,D\subseteq P_1$ and $B\subseteq P_2$,
and no part of $\mathbf{T}$ has nonempty intersection with both $P_1$ and $P_2$.

\smallskip

\textit{Type 2.3.} $A,B,D\subseteq P_1$ and $C\subseteq P_2$,
and no part of $\mathbf{T}$ has nonempty intersection with both $P_1$ and $P_2$.
  
\smallskip

\textit{Type 2.4.} $B,C,D\subseteq P_1$ and $A\subseteq P_2$,
and no part of $\mathbf{T}$ has nonempty intersection with both $P_1$ and $P_2$.

\medskip

\noindent\textit{Type 3.} $A,B,C,D\subseteq P_1$, and no part of
$\mathbf{T}$ has nonempty intersection with both $P_1$ and $P_2$.

\medskip

\noindent \textbf{Caution.} It is tempting to use the $S_4$-action that permutes $A,B,C,D$, but one must be very careful in doing so. The sets $A,B,C,D$ have been fixed, and may have different cardinalities; we may only invoke symmetry if our arguments do not refer to any specific properties of $A,B,C,D$. 
% Further, \eqref{eq:WPTilde1} incorporates the cardinalities of $A,B,C,D$ in a very subtle way: the requirements that $\abs{P_1}\ge\alpha_1+2$ and $\abs{P_2}\ge\alpha_2+2.$

\medskip

\noindent\textbf{If $\Gamma$ is of type 0, the coefficient of $\Gamma$ in $\tWPQn({\mathcal{R}})$ is zero.} By \eqref{eq:WPTilde1} and
\eqref{eq:WPTilde23}, no terms of ${\mathcal{R}}$ contribute a term of type 0 to 
$\tWPQn({\mathcal{R}})$. % $(A_rB_r\mathbf{U}_1|C_rD_r\mathbf{U}_2)$ is
% of type 0.
% If
% $\mathbf{U}_1,\mathbf{U}_2\ne\emptyset,$ then
% $$\tWPQn(A_rB_r\mathbf{U}_1|C_rD_r\mathbf{U}_2)% =W_{n,0}(A_rB_r\mathbf{U}_1|C_rD_r\mathbf{U}_2)
% =\left\{\textstyle \left(
%   % \begin{array}{c|c}
%   %   A\cup B\cup\bigcup_{T\in\mathbf{U}_1}T&C\cup D\cup\bigcup_{T\in\mathbf{U}_2}T\\\hline
%   %   \abs{\mathbf{U}_1}&\abs{\mathbf{U}_2}
%                           %     \end{array}
%     A_r\cup B_r\cup\bigcup_{a\in\mathbf{U}_1}a,\abs{\mathbf{U}_1}
%   \right),\left(C_r\cup
%     D_r\cup\bigcup_{a\in\mathbf{U}_2}a,\abs{\mathbf{U}_2}\right)\right\},$$
% so $A_r,B_r,C_r,D_r$ are not split up in any term of
% $\tWPQn(r)$, nor is any part of $\mathbf{T}$. If
% $\mathbf{U}_1=\emptyset$, then each term of
% $\tWPQn(A_rB_r\mathbf{U}_1|C_rD_r\mathbf{U}_2)$
% has $P_1$ and $P_2$ equal to unions of parts of the set partition
% $\{A_r\cup B_r\}\cup\{C_r\}\cup\{D_r\}\cup\mathbf{U}_2.$ Thus
% $A_r,B_r,C_r,D_r$ and parts of $\mathbf{T}$ are not split up in any
% term of $r$ of this form. In other words, for $\Gamma$ of type 0, the
% coefficient of $\Gamma$ in $\tWPQn(r)$ is zero.

\medskip
  
\noindent\textbf{If $\Gamma$ is of type 1.1, the coefficient of $\Gamma$ in $\tWPQn({\mathcal{R}})$ is zero.} Let
$\Gamma$ be of type 1.1, so we may write
  \begin{align*}
    \Gamma=\left\{\textstyle\left(A\cup B\cup\bigcup_{S\in\mathbf{T}_1}S,\alpha_1\right),\left(C\cup D\cup\bigcup_{S\in\mathbf{T}_2}S,\alpha_2
    % \begin{array}{c|c}
    %   A\cup B\cup\bigcup_{T\in\mathbf{T}_1}T&C\cup D\cup\bigcup_{T\in\mathbf{T}_2}T\\\hline
    %   \alpha_1&\alpha_2
    % \end{array}
                          \right)\right\}
  \end{align*}
  for some partition $\mathbf{T}_1\sqcup\mathbf{T}_2=\mathbf{T},$ and
  for some $\alpha_1,\alpha_2$. Note that
  $\alpha_1+\alpha_2=k=\abs{\mathbf{T}}=\abs{\mathbf{T}_1}+\abs{\mathbf{T}_2}.$
%   Using Symmetry 2 above, we may assume without loss of generality
%   that $\alpha_1\ge\abs{\mathbf{T}_1}.$
  Let $$\ell=\alpha_1-\abs{\mathbf{T}_1}=\abs{\mathbf{T}_2}-\alpha_2.$$

  \noindent \textbf{Claim: If $\ell=0,$ the coefficient of $\Gamma$ in
    $\tWPQn({\mathcal{R}})$ is zero.}
  \begin{proof}[Proof of claim]
    First, note that in this case $\abs{\mathbf{T}_1}=\alpha_1\ge1,$
    i.e. $\mathbf{T}_1\ne\emptyset.$ Similarly
     $\mathbf{T}_2\ne\emptyset.$ By definition, every term of $\mathcal{R}$ is of the form $(AB\mathbf{U}_1|CD\mathbf{U}_2)$ or $-(AC\mathbf{U}_1|BD\mathbf{U}_2)$, for some partition $\mathbf{U}_1\sqcup\mathbf{U}_2=\mathbf{T}$. 
    In terms of the latter form, $\Gamma$ appears with coefficient zero by \eqref{eq:WPTilde1} and \eqref{eq:WPTilde23}.
    % Consider terms of the latter form. If $\mathbf{U}_1,\mathbf{U}_2\ne\emptyset,$ then by $\eqref{eq:WPTilde23},$ $\tWPQn(AC\mathbf{U}_1|BD\mathbf{U}_2)\in\Pairs^2_{k,n}$ and $\tWPQn(AC\mathbf{U}_1|BD\mathbf{U}_2)\ne\Gamma.$ If $\mathbf{U}_1=\emptyset,$ then by \eqref{eq:WPTilde1}, $A$ and $C$ appear in the same subset in every term of $\tWPQn(AC\mathbf{U}_1|BD\mathbf{U}_2)$, so $\Gamma$ does not appear. (Similarly if $\mathbf{U}_2=\emptyset.$)
    In a term $(AB\mathbf{U}_1|CD\mathbf{U}_2)$ with $\mathbf{U}_1,\mathbf{U}_2\ne\emptyset$ and $\mathbf{U}_1\ne\mathbf{T}_1,$ $\Gamma$ also appears with coefficient zero by \eqref{eq:WPTilde23}. Thus the only terms of ${\mathcal{R}}$ that contribute to the coefficient of
    $\Gamma$ in $\tWPQn({\mathcal{R}})$ are:
    \begin{align}\label{eq:3terms}
      (AB\mathbf{T}_1|CD\mathbf{T}_2),&&(AB|CD\mathbf{T}),&&\text{and}&&(AB\mathbf{T}|CD).
    \end{align}
    By definition,
    $$\tWPQn((AB\mathbf{T}_1|CD\mathbf{T}_2))=1\cdot\Gamma.$$   % In which terms of $\tWPQn(r)$ does
    % \begin{align*}
    %   \Gamma=\left(
    %   \begin{array}{c|c}
    %     A\cup B\cup\bigcup_{T\in\mathbf{T}_1}T&C\cup D\cup\bigcup_{T\in\mathbf{T}_2}T\\\hline
    %     \abs{\mathbf{T}_1}&\abs{\mathbf{T}_2}
    %   \end{array}
    %                         \right)
    % \end{align*}
    % appear? It appears with coefficient 1 in the term
    % $\tWPQn(D(AB\mathbf{T}_1|CD\mathbf{T}_2)).$ There are two
    % other terms where it can appear, $\tWPQn(D(AB|CD\mathbf{T}))$
    % and $\tWPQn(D(AB\mathbf{T}|CD)).$
    The boundary stratum $(AB|CD\mathbf{T})$ has a single vertex with valence $\ge4,$ with corresponding set partition
    $\Pi=\{A\cup B,C,D\}\cup\mathbf{T}\vdash\{1,\ldots,n\}$. Using the notation in the paragraph preceding \eqref{eq:WPTilde1} (with respect to our fixed $\Gamma\in(\Pairs^2_{k,n})^0$), we have $s_1=1+\abs{\mathbf{T_1}}=1+\alpha_1$, so $s_1-\alpha_1=1.$ Thus $s_2-\alpha_2=2$, so $e_\Pi(\Gamma)=s_1-\alpha_1=1.$ The coefficient of $\Gamma$ in $\tWPQn((AB|CD\mathbf{T}))$ is thus $-1/2.$
    
    % By \eqref{eq:WPTilde1}, the coefficient of
    % $\Gamma$ in $\tWPQn((AB|CD\mathbf{T}))$ is (CONTINUE!) could appear in two summands of
    % $\tWPQn((AB|CD\mathbf{T}))$: the term where
    % $\Omega=\{A\cup B\}\cup\mathbf{T}_1$, and the term where
    % $\Omega=\{C,D\}\cup\mathbf{T}_2$. The latter case, however does not
    % appear, as $\abs{\Omega}-1+j=\alpha_2$ would imply $j=-1$.

    % In the former case, we must have $j=0$. note that
    % $2\le\abs{\Omega}\le k$ as needed, because
    % $\mathbf{T}_1,\mathbf{T}_2\ne\emptyset$. Also, because $A$ and $B$
    % are nonempty, we have $\abs{\Omega}+1\le\sum_{a\in \Omega}\abs{a}$. Thus
    % $(AB|CD\mathbf{T})$ contributes $-1/2$ to the coefficient of
    % $\Gamma$ in $\tWPQn({\mathcal{R}})$.
    
    The set partition associated to $(AB\mathbf{T}|CD)$ is
    $\Pi=\{A,B,C\cup D\}\cup\mathbf{T}$. In this case we have $s_1=2+\abs{\mathbf{T_1}}=2+\alpha_1$, so $s_1-\alpha_1=2.$ Thus $s_1-\alpha_1=1$, so $e_\Pi(\Gamma)=s_2-\alpha_2=1.$ The coefficient of $\Gamma$ in $\tWPQn((AB\mathbf{T}|CD))$ is thus $-1/2.$

    % By an identical argument, we also conclude that
    % $(AB\mathbf{T}|CD)$ contributes $-1/2$ to the coefficient of
    % $\Gamma$ in $\tWPQn({\mathcal{R}})$.
    % we have
    % $\SP(v)=\{A,B,C\cup D\}\cup\mathbf{T}$. There is a nonzero
    % coefficient of $\Gamma$ in
    % $\tWPQn(D(AB|CD\mathbf{T}))$ exactly when
    % $s=\{C\cup D\}\cup\mathbf{T}_2$ and $j=0$. (Again, $\abs{s}$
    % does
    % satisfies $2\le\abs{s}\le k$ and
    % $\abs{s}+1\le\sum_{a\in s}\abs{a}$.) (Note that when
    % $s=\{A,B\}\cup\mathbf{T}_1$, a nonzero coefficient of $\Gamma$
    % could only occur if $\abs{s}-1+j=\abs{\mathbf{T}_1},$
    % i.e. $j=-1.$)
    % Thus
    % $\tWPQn(D(AB\mathbf{T}|CD))$ contributes
    % -1/2
    % to the coefficient of $\Gamma$ in
    % $\tWPQn(r)$.
    In total, the coefficient is $1-1/2-1/2=0,$ proving the claim.
  \end{proof}

    \noindent \textbf{Claim: If $\ell\ne0,$ the coefficient of $\Gamma$ in
    $\tWPQn({\mathcal{R}})$ is zero.}
  \begin{proof}[Proof of claim]
    Note that $\alpha_1,\alpha_2\ge1$ imply $-\abs{\mathbf{T}_1}+1\le\ell\le k-\abs{\mathbf{T}_2}-1.$ %In which terms of
    % $\tWPQn(r)$ does
    % \begin{align*}
    %   \Gamma=\left(
    %   \begin{array}{c|c}
    %     A\cup B\cup\bigcup_{T\in\mathbf{T}_1}T&C\cup D\cup\bigcup_{T\in\mathbf{T}_2}T\\\hline
    %     \abs{\mathbf{T}_1}+\ell&\abs{\mathbf{T}_2}-\ell
    %   \end{array}
    %                         \right)
    % \end{align*}
    % appear? 
    By the same argument as the one preceding \eqref{eq:3terms}, the only terms of ${\mathcal{R}}$ that contribute to the coefficient of
    $\Gamma$ in $\tWPQn({\mathcal{R}})$ are
    \begin{align*}
      (AB\mathbf{T}|CD)&&\text{and}&&(AB|CD\mathbf{T}).
    \end{align*}
    (Note that $(AB\mathbf{T}_1|CD\mathbf{T}_2)$ does not contribute because the values of $\alpha_1,\alpha_2$ in $\tWPQn((AB\mathbf{T}_1|CD\mathbf{T}_2))$ do not match those in $\Gamma$.
    % Note also that we cannot
    % use symmetry to show that the contributions from these two terms are equal, as we broke
    % symmetry by assuming $\alpha_1>\abs{\mathbf{T}_1}$ and
    % $\alpha_2<\abs{\mathbf{T}_2}.$)

    As above, $(AB|CD\mathbf{T})$ induces set partition
    $\Pi_1=\{A\cup B,C,D\}\cup\mathbf{T}.$ Using the notation in the paragraph preceding \eqref{eq:WPTilde1} (with respect to our fixed $\Gamma\in(\Pairs^2_{k,n})^0$), we have $s_1=1+\abs{\mathbf{T_1}}=1+\alpha_1-\ell$, so $s_1-\alpha_1=1-\ell$ (and $s_2-\alpha_2=2+\ell$). Thus $e_{\Pi_1}(\Gamma)=\min\{1-\ell,2+\ell\}$.
    %, and the coefficient of $\Gamma$ in $\tWPQn((AB|CD\mathbf{T}))$ is $(-1)^{1-\ell}/2.$
    
    On the other hand, $(AB\mathbf{T}|CD)$ induces set partition
    $\Pi_2=\{A,B,C\cup D\}\cup\mathbf{T}.$ In this case $s_1=2+\abs{\mathbf{T_1}}=2+\alpha_1-\ell$, so $s_1-\alpha_1=2-\ell\le 1$ (and $s_2-\alpha_2=1+\ell\ge2$). Thus $e_{\Pi_2}(\Gamma)=\min\{1+\ell,2-\ell\}$
    %, and the coefficient of $\Gamma$ in $\tWPQn((AB\mathbf{T}|CD))$ is $(-1)^{2-\ell}/2.$
    
    If $\ell>0,$ the coefficient of $\Gamma$ in $\tWPQn(\mathcal{R})$ is $(-1)^{1-\ell}/2+(-1)^{2-\ell}/2=0.$ If $\ell<0,$ the coefficient of $\Gamma$ in $\tWPQn(\mathcal{R})$ is $(-1)^{2+\ell}/2+(-1)^{1+\ell}/2=0.$ This proves the claim.
  \end{proof}
  We conclude that the coefficient of $\Gamma$ in
  $\tWPQn({\mathcal{R}})$ is zero.

  \medskip

  \noindent\textbf{If $\Gamma$ is of type 1.2, the coefficient of $\Gamma$ in $\tWPQn({\mathcal{R}})$ is zero.} The argument for the type 1.1 case did not refer to any properties of the sets $A,B,C,D$, so this case follows by symmetry.
  
%   This is immediate from
%   the type 1.1 case, using Symmetry 1.
  % $\Gamma$ is obtained from some $\Gamma'$ of type 1.1 via the
  % operation appearing in Symmetry 2. Thus the coefficient of $\Gamma$
  % in $\tWPQn(r)$ is equal to the coefficient of $\Gamma'$ in
  % $\tWPQn(-r)=-\tWPQn(r)$,
  % which we just showed was zero.

  \medskip

  \noindent\textbf{If $\Gamma$ is of type 1.3, the coefficient of $\Gamma$ in $\tWPQn({\mathcal{R}})$ is zero.}  No terms of ${\mathcal{R}}$ contribute a term of type 1.3
  to $\tWPQn({\mathcal{R}})$.

  \bigskip

  \noindent\textbf{If $\Gamma$ is of type 2.1, the coefficient of $\Gamma$ in $\tWPQn({\mathcal{R}})$ is zero.} The only terms of $\mathcal{R}$ that contribute to the coefficient of $\Gamma$ in $\tWPQn(\mathcal{R})$ are:
  \begin{align*}
      (AB|CD\mathbf{T})&&\text{and}&&-(AC|BD\mathbf{T}).
  \end{align*}
  The corresponding set partitions are $\Pi_1=\{A\cup B,C,D\}\cup\mathbf{T}$ and $\Pi_1=\{A\cup C,B,D\}\cup\mathbf{T}$. By definition, $e_{\Pi_1}(\Gamma)=e_{\Pi_2}(\Gamma)$, so the coefficients of $\Gamma$ in $\tWPQn((AB|CD\mathbf{T}))$ and $\tWPQn(-(AC|BD\mathbf{T}))$ cancel.
  % This type consists of $\Gamma$ of the form
  % \begin{align*}
  %   % \left(
  %   % \begin{array}{c|c}
  %   %   A\cup B\cup C\cup\bigcup_{T\in\mathbf{T}_1}T&D\cup\bigcup_{T\in\mathbf{T}_2}T\\\hline
  %   %   \alpha_1&\alpha_2
  %   % \end{array}
  %   % \right)
  %                 \{\textstyle\left(A\cup B\cup C\cup\bigcup_{T\in\mathbf{T}_1}T,\alpha_1\right),\left(D\cup\bigcup_{T\in\mathbf{T}_2}T,\alpha_2\right)\}
  % \end{align*}
  % for some partition $\mathbf{T}_1\sqcup\mathbf{T}_2=\mathbf{T},$ and
  % for some $\alpha_1,\alpha_2$.
  
%   The operation of transposing $B$ and $C$ (as in Symmetry 1) sends
%   ${\mathcal{R}}\mapsto-{\mathcal{R}}.$ However, it preserves the coefficient of $\Gamma$ in
%   $\tWPQn({\mathcal{R}})$, since $\Gamma$ is invariant under
%   the operation. Thus the coefficient of $\Gamma$ in
%   $\tWPQn({\mathcal{R}})$ is zero.

  \medskip

  \noindent\textbf{If $\Gamma$ is of type 2.2, 2.3, or 2.4, the coefficient of $\Gamma$ in $\tWPQn({\mathcal{R}})$ is zero.} The argument for type 2.1 did not refer to any properties of the sets $A,B,C,D$, so these cases follow by symmetry.

  % \smallskip

  % \noindent\textbf{$\tWPQn(r)$ has zero coefficient of any $\Gamma$ of types 2.3 and
  %   2.4.} These are obtained from $\Gamma'$ of types 2.2 and 2.1 by
  % Symmetry 1.

  \medskip

  \noindent\textbf{If $\Gamma$ is of type 3, the coefficient of $\Gamma$ in $\tWPQn({\mathcal{R}})$ is zero.} The only terms of $\mathcal{R}$ that contribute to the coefficient of $\Gamma$ in $\tWPQn(\mathcal{R})$ are:
  \begin{align*}
      (AB|CD\mathbf{T}),&&(AB\mathbf{T}|CD),&&-(AC|BD\mathbf{T}),&&\text{and}&&-(AC\mathbf{T}|BD).
  \end{align*}
  The corresponding set partitions are $\Pi_1=\{A\cup B,C,D\}\cup\mathbf{T}$, $\Pi_2=\{A,B,C\cup D\}\cup\mathbf{T}$, $\Pi_3=\{A\cup C,B,D\}\cup\mathbf{T}$, and $\Pi_1=\{A,C,B\cup D\}\cup\mathbf{T}$. By definition, $e_{\Pi_1}(\Gamma)=e_{\Pi_2}(\Gamma)=e_{\Pi_3}(\Gamma)=e_{\Pi_4}(\Gamma)$, so the coefficients of $\Gamma$ in $\tWPQn((AB|CD\mathbf{T})),$ $\tWPQn((AB\mathbf{T}|CD)),$ $\tWPQn(-(AC|BD\mathbf{T})),$ and $\tWPQn(-(AC\mathbf{T}|BD))$ cancel.
%   Again,
%   $\Gamma$ is invariant under transposing $B$ and $C$, so its
%   coefficient in $\tWPQn({\mathcal{R}})$ is zero.

  Altogether, we conclude that $\tWPQn({\mathcal{R}})=0.$ This
  completes Case VI and the proof of the base case.
\end{proof}
  % \medskip

  % \noindent\textbf{Base Case:} Let $b=0.$ ...

Next we prove the inductive step.

  % \medskip
  
\begin{prop}\label{prop:PairsBIsom}
  Suppose
  $W_{n,b}:\Q(\mathbf{S}^2_{k,n})^{b}\to\Q(\Pairs^2_{k,n})^{b}$
  descends $S_n$-equivariantly to $(Q^2_{k,n})^b$ for a fixed $b$ and $n$. Then
  $W_{n,b+1}:\Q(\mathbf{S}^2_{k,n+1})^{b+1}\to\Q(\Pairs^2_{k,n+1})^{b+1}$
  descends $S_{n+1}$-equivariantly to $(Q^2_{k,n+1})^{b+1}$. In particular, if Lemma \ref{lem:QToP} holds
  for $(b,n)$ for all $n$ (and fixed $b$), then it holds for $(b+1,n)$
  for all $n$.
\end{prop}
\begin{proof}
  Suppose
  $W_{n,b}:\Q(\mathbf{S}^2_{k,n})^{b}\to\Q(\Pairs^2_{k,n})^{b}$
  descends equivariantly to $(Q^2_{k,n})^b$. Let $\pi_{n+1}$ be the forgetful map that
  forgets the last point. Note that
  $(\pi_{n+1})_*((Q^{2}_{k,n+1})^{\ge b+1})\subseteq
  (Q^{2}_{k,n})^{\ge b}.$ (This holds because for any tree in $(\mathbf{S}^{2}_{k,n+1})^{\ge b+1}$, the corresponding boundary stratum is either contracted
  in dimension by $\pi_{n+1}$, or maps isomorphically to a boundary
  stratum in $(\mathbf{S}^{2}_{k,n})^{\ge b}$.) Thus there is an
  induced map on quotients
  \begin{align*}
    (\pi_{n+1})_*:(Q^{2}_{k,n+1})^{b+1}\to (Q^{2}_{k,n})^{b}.
  \end{align*}
  By the inductive hypothesis, there is a map
  \begin{align*}
    \bar{W_{n,b}}:(Q^{2}_{k,n})^{b}\to(\Pairs^2_{k,n})^{b},
  \end{align*}
  such that the lower-right triangle commutes in the following
  diagram:
\begin{center}
\begin{tikzcd}
  % R_{2,k,n+1}^{\ep=c}\arrow[rr,hook]&
  &\Q(\mathbf{S}^{2}_{k,n+1})^{b+1}\arrow[d,"\rho_{n+1,b+1}",swap]\arrow[drr,"W_{n+1,b+1}"]\arrow[dddr,"\pi_{n+1,b+1}"]&\\
  % &
    &(Q^{2}_{k,n+1})^{b+1}\arrow[dddr,"(\pi_{n+1})_*",swap]\arrow[rr,dashed,bend
    right=25]&&\Q(\Pairs^2_{k,n+1})^{b+1}\arrow[dddr,"\pi_{n+1,b+1}"]\\
    &\\
  &&\Q(\mathbf{S}^{2}_{k,n})^{b}\arrow[d,"\rho_{n,b}"]\arrow[drr,"W_{n,b}"]\\
  % &
  &&(Q^{2}_{k,n})^{b}\arrow[rr,"\bar{W_{n,b}}",swap]&&\Q(\Pairs^2_{k,n})^{b}
\end{tikzcd}
\end{center}
Here the map
$\pi_{n+1,b+1}:\Q(\mathbf{S}^{2}_{k,n+1})^{b+1}\to\Q(\mathbf{S}^{2}_{k,n})^{b}$
sends a tree $\sigma$ to the tree obtained by contracting the $(n+1)$-st marked half-edge and stabilizing (see e.g. \cite{KockVainsencher2006}) if the $(n+1)$-st marked half-edge is
on $\sigma'$ (using notation from Definition \ref{Def:WPDef}),
and to zero otherwise. The map
$\pi_{n+1,b+1}:\Q(\Pairs^{2}_{k,n+1})^{b+1}\to\Q(\Pairs^{2}_{k,n})^{b}$
sends a pair $\{(P_1,\alpha_1),(P_2,\alpha_2)\}$ to itself if
$n+1\in(P_1\cup P_2)^C$, and to zero otherwise. The entire diagram is
easily checked to commute, essentially using the fact that if
$\sigma\in(\mathbf{S}^{2}_{k,n+1})^{b+1}$ is such that
$n+1\in\sigma_1,$ then
$(\pi_{n+1})_*(\rho_{n+1,b+1}(\sigma))\in(Q^2_{k,n})^{\ge b+1}.$ We are
trying to prove the existence of the dashed arrow.

% \medskip

% \textbf{Claim.} \textit{This diagram commutes.} Fix a boundary stratum
% $\sigma\in(\mathbf{S}^{2}_{k,n+1})^{b+1}.$ Write
% $W_{n,b}(\sigma)=\{(P_1,\alpha_1),(P_2,\alpha_2)\}.$

% If $n+1\in P_1,$ then
% $(\pi_{n+1})_*(W_{n+1,b+1}(\sigma))=0$. In this case, either
% $n+1$ is on a component with moduli (in which case $(\pi_{n+1})_*$
% sends $\sigma\mapsto0$) or it is not, in which case
% $(\pi_{n+1})_*(\sigma)\in (Q^{2}_{k,n})^{\ep\le c-1}.$ But then it is
% zero in the quotient $(Q^{2}_{k,n})^{\ep=c}$, so in either case the
% ``left path'' gives zero. Similarly if $n+1\in P_2$.

% If $n+1\not\in P_1\cup P_2,$ then the right path gives $W(\sigma),$
% viewed as an element of $\Q(\Pairs^2_{k,n})^{\ep=c}$ rather than
% $\Q(\Pairs^2_{k,n+1})^{\ep=c}$. The left path gives
% $W((\pi_{n+1})_*(\sigma)),$ which is the same. This proves the
% claim.

% Consider a relation $R\in R_{2,k,n+1}^{\ep=c}.$
Let $r\in\Q(\mathbf{S}^{2}_{k,n+1})^{b+1}$ be such that
$\rho_{n+1,b+1}(r)=0.$ Then
certainly
$$\bar{W_{n,b}}((\pi_{n+1})_*(\rho_{n+1,b+1}(r)))=0,$$ so by
commutativity, $W_{n+1,b+1}(r)$ lies
in
$$\ker(\pi_{n+1,b+1})=\Span\{\{(P_1,\alpha_1),(P_2,\alpha_2)\}\in(\Pairs^2_{k,n+1})^{b+1}:n+1\in
P_1\cup P_2\}.$$ Then by $S_{n+1}$-symmetry, $W_{n+1,b+1}(r)$ actually
lies in the smaller
subspace
$$\Span\{\{(P_1,\alpha_1),(P_2,\alpha_2)\}\in(\Pairs^2_{k,n+1})^{b+1}:P_1\cup
P_2=\{1,\ldots,n+1\}\}.$$ Since $b+1>0,$ this subspace is trivial, so
$W_{n+1,b+1}(r)=0.$ Thus $W_{n+1,b+1}$ descends to a
(clearly $S_{n+1}$-equivariant) map
$\bar{W_{n+1,b+1}}:(Q^{2}_{k,n+1})^{b+1}\to\Q(\Pairs^2_{k,n+1})^{b+1}.$ This completes the proof, and we conclude Lemma \ref{lem:QToP}.
\end{proof}
We have just shown that $W_{n,b}$ descends to $(Q^2_{k,n})^b$. Next we must prove that $\rho_{n,b}$ descends to $(\Pairs^2_{k,n})^b$. We will use the following well-known fact.
\begin{prop}\label{prop:RearrangeFullyDegenerateSubtree}
  Let $\sigma\in\mathbf{S}_{k,n},$ and suppose $\sigma'$ can be
  obtained from $\sigma$ by rearranging a trivalent
  subtree. (That is, suppose there exist trivalent subtrees of
  $\sigma$ and $\sigma'$ whose complements are isomorphic as marked
  trees, and such that the two subtrees have the same marking set.)
  Then $\sigma$ and $\sigma'$ have the same class in
  $H_{2k}(\Mbar_{0,n}).$
\end{prop}
\begin{proof}[Sketch of proof]
  Contracting the relevant trivalent subtrees, $\sigma$ and $\sigma'$
  can both be written as the pushforward from the same product
  $\Mbar_{0,n_0}\times\Mbar_{0,n_1}\times\cdots\times\Mbar_{0,n_m},$
  of the class $[\mathrm{pt}]\times1\times\cdots\times1,$ for
  different choices of $\mathrm{pt}\in\Mbar_{0,n_0},$ the factor
  corresponding to the contracted subtrees.
\end{proof}

\begin{lem}\label{lem:PToQ}
  Fix nonnegative integers $n\ge4$, $k\le n-3,$ and $b\le n-k-4.$ The
  map $\rho_{n,b}:\Q(\mathbf{S}^2_{k,n})^b\to(Q^2_{k,n})^b$ descends
  equivariantly to a map
  $\bar{\rho_{n,b}}:\Q(\Pairs^2_{k,n})^b\to(Q^2_{k,n})^b$.
\end{lem}
\begin{proof}
  We must check that $\rho_{n,b}(\ker(W_{n,b}))=0.$ Note that as an induced map on free vector spaces, $W_{n,b}$ has kernel spanned by $$\{\sigma-\sigma':\sigma,\sigma'\in\mathbf({S}^2_{k,n})^b,W_{n,b}(\sigma)=W_{n,b}(\sigma')\}.$$ Thus it suffices to show that if $\sigma,\sigma'\in\mathbf({S}^2_{k,n})^b$ and $W_{n,b}(\sigma)=W_{n,b}(\sigma')$, then $\rho_{n,b}(\sigma)=\rho_{n,b}(\sigma')$.
  We consider two cases: $b=n-k-4$ and $b<n-k-4$. 
  
  First, suppose $b=n-k-4.$ 
%   Then $(Q^2_{k,n})^{\ge b+1}=0,$ so
%   $$(Q^2_{k,n})^b=(Q^2_{k,n})^{\ge
%     b}=\Span((\mathbf{S}^2_{k,n})^{b})\subseteq H_{2k}(\Mbar_{0,n}).$$
  Then for any $\sigma\in(\mathbf{S}^2_{k,n})^{\ge b}$, the two vertices
  $v_1$ and $v_2$ are leaves, and $W_{n,b}(\sigma)$ determines
  $\sigma$ up to rearrangement of the trivalent subtree
  obtained by deleting $v_1$ and $v_2.$ Thus $\rho_{n,b}$ descends to $(\Pairs^2_{k,n})^b$ by
  Proposition \ref{prop:RearrangeFullyDegenerateSubtree}.
  Next, suppose $b<n-k-4.$ Let $\tau\in(\mathbf{S}^2_{k,n})^{b}.$ We will next define an algorithm (\eqref{eq:alg} below) for producing trees $\tau'\in(\mathbf{S}^2_{k,n})^{b}$ such that $\rho_{n,b}(\tau')=\rho_{n,b}(\tau)$. Note that Proposition \ref{prop:RearrangeFullyDegenerateSubtree} is another such algorithm. We will then prove that if $\sigma,\sigma'\in\mathbf({S}^2_{k,n})^b$ and $W_{n,b}(\sigma)=W_{n,b}(\sigma')$, then $\sigma'$ can be obtained from $\sigma$ via these two algorithms.

  Let $\tau\in(\mathbf{S}^2_{k,n})^{\ge b}$. Since $b<n-k-4$, at least one of  $v_1$ and $v_2$ is not a leaf. Suppose without loss of
  generality that $v_1$ is not a leaf. Let $e_1$ be the edge incident
  to $v_1$ that is on the path from $v_1$ to $v_2.$ Fix an edge
  $e\ne e_1$, and write $e=\{v_1,v'\}$. Let $\bar{\tau}$ be the tree obtained by contracting $e$ to a vertex $v_1=v'$. We consider
  Kontsevich-Manin relations $\mathcal{R}=\mathcal{R}(\bar{\tau},v_1=v',A,B,C,e_1),$ where $A,B\ne e$ are incident to $v'$ in $\tau$, and $C\ne e$ is incident to $v_1$ in $\tau$.
  All but four terms of $\mathcal{R}$ are in
  $\mathbf{S}^{\ge3}_{k,n}.$ The remaining terms are:
  \begin{align*}
    (AB\mathbf{T}|Ce_1)+(AB|Ce_1\mathbf{T})-(AC\mathbf{T}|Be_1)+(AC|Be_1\mathbf{T})
  \end{align*}
  where $\mathbf{T}$ is the set of edges and half-edges incident to
  $v_1=v'$ other than $A,B,C,e_1$. Note that the first and third terms are in
  $(\mathbf{S}^{2}_{k,n})^{\ge b+1},$ hence are zero in
  $(Q^2_{k,n})^b.$ We are left with the relation
  \begin{align}\label{eq:alg}
    (AB|Ce_1\mathbf{T})=(AC|Be_1\mathbf{T}).
  \end{align}
  The left side is $\tau$, and the right side is another element of $(\mathbf{S}^2_{k,n})^b$.
  
%   We now show that if
%   $\sigma,\sigma'\in(\mathbf{S}^2_{k,n})^b$ are such that
%   $W_{n,b}(\sigma)=W_{n,b}(\sigma'),$ then there exists a sequence
%   of relations of the form \eqref{eq:alg} (together with applications
%   of Proposition \ref{prop:RearrangeFullyDegenerateSubtree}) that
%   begins at $\sigma$ and ends at $\sigma'.$

  Fix $P=\{(P_1,\alpha_1),(P_2,\alpha_2)\}\in(\mathbf{P}^2_{k,n})^b.$ We must now show that if $\sigma,\sigma'\in(W_{n,b})^{-1}(P),$ then $\sigma'$ can be obtained from $\sigma$ via the two algorithms \eqref{eq:alg} and Proposition \ref{prop:RearrangeFullyDegenerateSubtree}.
  We will instead show that any $\sigma\in(W_{n,b})^{-1}(P)$ can
  be manipulated via \eqref{eq:alg} into a standard form. That is, we will find $\sigma_0\in(W_{n,b})^{-1}(P)$ such that for any $\sigma\in(W_{n,b})^{-1}(P)$, $\sigma_0$ can be obtained from $\sigma$ via the two algorithms. Let $\sigma_0$ be as follows: On $v_1$, the first $\alpha_1+1$
  elements of $P_1$ (in increasing order) are half-edges incident to
  $v_1,$ and the remaining elements of $P_1$ form a trivalent
  subtree incident to $v_1$. Similarly for $v_2,$ and the remaining
  points $(P_1\cup P_2)^C$ of course form a trivalent subtree
  connecting $v_1$ and $v_2$. This defines $\sigma_0$ up to rearrangement of trivalent subtrees; this is sufficient in light of Proposition
  \ref{prop:RearrangeFullyDegenerateSubtree}.

  Fix $\sigma\in(W_{n,b})^{-1}(P)$.  If the first $\alpha_1+1$
  marked half-edges of $P_1$ are incident to $v$, then $v_1$ is already in the desired form. If not,
  let $e,\ne e_1$ be incident to $v_1$ such that subtree corresponding
  to $e=\{v_1,v'\}$ contains (at least) one of the first $\alpha_1+1$ marked
  half-edges. Using Proposition
  \ref{prop:RearrangeFullyDegenerateSubtree}, rearrange the subtree
  so that this half-edge is incident to $v',$ then apply \eqref{eq:alg}, where
  $B$ is the chosen half-edge, and $C$ is any edge or half-edge
  incident to $v_1$ that is not equal to $e$, nor to any of the first
  $\alpha_1+1$ marked half-edges. (Such a flag must exist, since
  there are $\alpha_1+1$ edges and half-edges other than $e$ and
  $e_1$, and we have assumed that one of the first $\alpha_1+1$ marked
  half-edges is on the subtree corresponding to $e$.) In the resulting
  tree, the number of the first $\alpha_1+1$ marked
  half-edges incident to $v_1$ has increased by 1. We repeat until this number is $\alpha_1+1$, then apply the same argument to $v_2$. The result is $\sigma_0$; this completes the proof.
  % that comes
  % from smoothing the node $\xi(e)$ and forgetting all but four
  % special points on the resulting vertex, where none of the special
  % points is the edge towards $v_2.$ In this relation, there are
  % three types of boundary strata. There
  % are strata with moduli on more than two components, which are zero
  % in $Q^{2}_{k,n}$. There are strata $S_{\sigma'}$ with
  % $\pointdistance(\sigma')=b+1,$ which are zero in
  % $(Q^{2}_{k,n})^b/(Q^{2}_{k,n})^{b+1}.$
  % Finally, there are two remaining terms, one on each side of the
  % equation. One of them is $S_\sigma,$ and the other is
  % $S_{\sigma'},$
  % where $\sigma'$ is obtained from $\sigma$ by taking a subtree
  % (possibly just a leg) attached to $v_1$ and a subtree attached to
  % $v',$ and switching their positions.
  % Note that this this relation is compatible with the data in the
  % statement of the theorem. Repeating the process, and using the
  % fact
  % that strata are homologous if related by $\SP$-preserving
  % bijections
  % between the set of vertices with moduli, we may reach any stratum
  % that shares that data.
\end{proof}

\section{A conjectural dimension formula}\label{sec:conj}
We end with a conjectural formula for the dimension of $Q^r_{k,n}.$
\begin{conj}\label{conj:DimConj}
The dimension of $Q^r_{k,n}$ is $$\sum_{\substack{\alpha_1,\ldots,\alpha_r\ge1\\\alpha_1+\cdots+\alpha_r=k}}\sum_{\substack{p_1,\ldots,p_r,b\ge0\\p_1+\cdots+p_r+b=n+r-2\\p_i\ge\alpha_i+2}}\frac{1}{r!}\binom{n+r-2}{p_1,\ldots,p_r,b}.$$
\end{conj}
The case $r=1$ appears in \cite{RamadasThesis}. The case $r=2$ follows from Theorem \ref{thm:MainThmStrong}, since the formula gives the cardinality of $\Pairs^2_{k,n}$. Summing over $r\in\{1,\ldots,\min\{k,n-2-k\}\}$ gives a conjectural closed formula for the dimension of $H_{2k}(\Mbar_{0,n},\Q)$. This formula agrees with the actual dimension for $n\le20$ and all $k$, but we have not been able to find it in the literature. (It is not obvious to us whether or not the formula is recoverable from e.g. \cite{Yuzvinsky1997}.)

\bibliographystyle{alpha}
\bibliography{research}

\begin{thebibliography}{Ram18}

\bibitem[BM13]{BergstromMinabe2013}
Jonas Bergstr{\"o}m and Satoshi Minabe.
\newblock On the cohomology of moduli spaces of (weighted) stable rational
  curves.
\newblock {\em Mathematische Zeitschrift}, 275(3-4):1095--1108, 2013.

\bibitem[CT20]{CastravetTevelev2020}
Ana-Maria Castravet and Jenia Tevelev.
\newblock Derived category of moduli of pointed curves -- {II}.
\newblock {\em ArXiv e-prints}, February 2020.
\newblock \href{http://arxiv.org/abs/2002.02889}{\texttt{arXiv:2002.02889}}.

\bibitem[FG03]{FarkasGibney2003}
Gavril Farkas and Angela Gibney.
\newblock The {M}ori cones of moduli spaces of pointed curves of small genus.
\newblock {\em Transactions of the American Mathematical Society},
  355:1183--1199, 2003.

\bibitem[FM94]{FultonMacpherson1994}
William Fulton and Robert MacPherson.
\newblock A compactification of configuration spaces.
\newblock {\em Annals of Mathematics}, 139(1):183--225, 1994.

\bibitem[Get95]{Getzler1995}
Ezra Getzler.
\newblock Operads and moduli spaces of genus 0 {R}iemann surfaces.
\newblock In Robbert~H. Dijkgraaf, Carel~F. Faber, and Gerard B.~M. van~der
  Geer, editors, {\em The Moduli Space of Curves}, pages 199--230, Boston, MA,
  1995. Birkh{\"a}user Boston.

\bibitem[Kap93]{Kapranov1993}
Mikhail Kapranov.
\newblock Chow quotients of {G}rassmannians {I}.
\newblock {\em Adv. Soviet Math}, 16(2):29--110, 1993.

\bibitem[Kee92]{Keel1992}
Sean Keel.
\newblock Intersection theory on the moduli space of stable $n$-pointed curves
  of genus zero.
\newblock {\em Transactions of the American Mathematical Society}, 330(2),
  1992.

\bibitem[KM94]{KontsevichManin1994}
Maxim Kontsevich and Yuri Manin.
\newblock Gromov-{W}itten classes, quantum cohomology, and enumerative
  geometry.
\newblock {\em Comm. Math. Phys.}, 164(3):525--562, 1994.

\bibitem[KV06]{KockVainsencher2006}
Joachim Kock and Israel Vainsencher.
\newblock {\em An Invitation to Quantum Cohomology: Kontsevich's Formula for
  Rational Plane Curves}.
\newblock Birkh{\"a}user, 2006.

\bibitem[Man95]{Manin1995}
Yuri Manin.
\newblock Generating functions in algebraic geometry and sums over trees.
\newblock In Robbert~H. Dijkgraaf, Carel~F. Faber, and Gerard B.~M. van~der
  Geer, editors, {\em The Moduli Space of Curves}, pages 401--417, Boston, MA,
  1995. Birkh{\"a}user Boston.

\bibitem[Ram17]{RamadasThesis}
Rohini Ramadas.
\newblock {\em Dynamics on the moduli space of pointed rational curves}.
\newblock PhD thesis, University of Michigan, 2017.

\bibitem[Ram18]{Ramadas2015}
Rohini Ramadas.
\newblock {Hurwitz correspondences on compactifications of
  $\overline{\mathcal{M}}_{0,N}$}.
\newblock {\em Advances in Mathematics}, 323:622--667, 2018.

\bibitem[Yuz97]{Yuzvinsky1997}
Sergey Yuzvinsky.
\newblock Cohomology bases for the {D}e {C}oncini--{P}rocesi models of
  hyperplane arrangements and sums over trees.
\newblock {\em Inventiones mathematicae}, 127(2):319--335, Jan 1997.

\end{thebibliography}

\end{document}